\documentclass[11pt]{amsart}
\usepackage{amscd,amssymb}
\usepackage{amsthm,amsmath,amssymb}
\usepackage[matrix,arrow,curve]{xy}
\usepackage{mathrsfs}
\usepackage{supertabular,longtable, rotating}
\usepackage{pdflscape}

\newcounter{cequation}[section]
\newtheorem{theorem}[cequation]{Theorem}

\newtheorem{proposition}[cequation]{Proposition}
\newtheorem{lemma}[cequation]{Lemma}
\newtheorem{corollary}[cequation]{Corollary}

\theoremstyle{definition}
\newtheorem{example}[cequation]{Example}

\theoremstyle{remark}
\newtheorem{remark}[cequation]{Remark}

\makeatletter\@addtoreset{equation}{section}
\@addtoreset{cequation}{section}

\makeatother

\newcommand{\tabfont}[1]{{}_{#1}}

\def\P {\mathbb{P}}

\def\C {\mathbb{C}}
\def\Q {\mathbb{Q}}
\def\WW {\mathbb{W}}

\def\II {\mathbb{I}}
\def\VV {\mathbb{V}}
\def\A {\mathfrak{A}}
\def\SS {\mathfrak{S}}
\def\BB {\mathcal{B}}
\def\LL {\mathcal{L}}

\def\PSp {\mathrm{PSp}_4(\mathbf{F}_3)}

\def\Cl {\mathrm{Cl}}
\def\rkCl {\mathrm{rk\,Cl}}

\def\vstrut {\vphantom{$\sqrt{A^A}$}}

\def\le {\leqslant}

\author[Cheltsov, Przyjalkowski, Shramov]{Ivan Cheltsov, Victor Przyjalkowski, and Constantin Shramov}

\title{Burkhardt quartic, Barth sextic, and the icosahedron}

\address{\emph{Ivan Cheltsov}
\newline
\textnormal{School of Mathematics, The University of Edinburgh.
}
\newline
\textnormal{National Research University Higher School of Economics, Russian Federation.
}
\newline
\textnormal{\texttt{I.Cheltsov@ed.ac.uk}}}

\address{\emph{Victor Przyjalkowski}
\newline
\textnormal{Steklov Mathematical Institute of Russian Academy of Sciences, Moscow, Russia.
}
\newline
\textnormal{National Research University Higher School of Economics, Russian Federation.
}
\newline
\textnormal{\texttt{victorprz@mi.ras.ru, victorprz@gmail.com}}}

\address{\emph{Constantin Shramov}
\newline
\textnormal{Steklov Mathematical Institute of Russian Academy of Sciences, Moscow, Russia.
}
\newline
\textnormal{National Research University Higher School of Economics, Russian Federation.
}
\newline
\textnormal{\texttt{costya.shramov@gmail.com}}}

\begin{document}

\begin{abstract}
We study two rational Fano threefolds with an action of the icosahedral group~$\A_5$.
The first one is the famous Burkhardt quartic threefold,
and the second one is the double cover of the projective space branched in the Barth sextic surface.
We prove that both of them are $\A_5$-Fano varieties that are $\A_5$-birationally superrigid.
This gives two new embeddings of the group $\A_5$ into the space Cremona group.
\end{abstract}

\sloppy

\maketitle

\section{Introduction}
\label{section:intro}

One of the interesting subjects in three-dimensional birational geometry is studying $G$-Fano threefolds (see \cite{Prokhorov1}),
where $G$ is a finite group.
These are Fano threefolds with terminal singularities such that the rank of their $G$-invariant class group equals~$1$.
Finding enough details about their $G$-equivariant (biregular) geometry often leads
to conclusions about the absence of $G$-equivariant birational maps between them,
see e.g.~\cite{ChSh09b} and \cite{CheltsovShramov}.
It is a usual case that on this way one encounters really beautiful geometric constructions arising from
large groups of symmetries of the corresponding varieties.

The icosahedral group $\A_5$ appears as a group of symmetries for a remarkably large class of Fano threefolds,
and many of them are $\A_5$-Fano threefolds, see e.g. \cite{CheltsovShramov}, \cite{CheltsovPrzyjalkowskiShramov}, and~\cite{PrzyalkowskiShramov}.
In this paper we study two rational Fano threefolds with an action of the group $\A_5$.
Both of these threefolds are singular and have only nodes as singularities;
moreover, they are ``extremal'' in some sense, which makes their geometry more interesting.
We prove that the ranks of their $\A_5$-invariant class groups equal~$1$, so that both of them are $\A_5$-Fano threefolds.
We use this to prove that they are $\A_5$-birationally superrigid, see \cite[\S3.1]{CheltsovShramov}
for a definition.

Consider the projective space $\P^5$ with homogeneous coordinates $x_0,\ldots,x_5$.
Denote by~\mbox{$\sigma_k(x_0,\ldots,x_5)$} the $k$-th elementary symmetric polynomial in $x_0,\ldots,x_5$.
The \emph{Burkhardt quartic} is defined in $\P^5$ by equations
$$
\sigma_1(x_0,\ldots,x_5)=\sigma_4(x_0,\ldots,x_5)=0.
$$
This quartic threefold was first described by Burkhardt in~\cite{Burkhardt}.
It has $45$ nodes, which is actually the largest possible number of isolated singularities for quartic threefolds, see~\cite{deJong} (cf.~\cite{Prokhorov-factorial}).

The Burkhardt quartic is known to be rational (see~\cite{Todd}), and its automorphism group is isomorphic to $\PSp$, see~\cite[\S6]{Coble} and~\cite{deJong}.
Let the group $\SS_6$ act on $\P^5$ by permutations of homogeneous coordinates $x_0,\ldots,x_5$. Then $\SS_6$ preserves the Burkhardt quartic.

Let us refer to a subgroup $\A_5$ in $\SS_6$ that fixes one of the homogeneous coordinates~\mbox{$x_0,\ldots,x_5$} as a \emph{standard} subgroup $\A_5$,
and to a subgroup $\A_5$ of $\SS_6$ that does not fix any of these  coordinates as a \emph{non-standard} subgroup $\A_5$. Our first result
is

\begin{theorem}
\label{theorem:Burkhardt}
Let $G$ be a standard subgroup $\A_5$ in $\SS_6$.
The Burkhardt quartic is a rational $G$-Fano threefold that is $G$-birationally superrigid.
\end{theorem}

It should be pointed out that if $G^\prime$ is a non-standard subgroup $\A_5$ in $\SS_6$, then the Burkhardt quartic is \emph{not} a $G^\prime$-Fano threefold (see Remark~\ref{remark:non-standard} below).

Now consider the projective space $\P^3$ that is a projectivization of a non-trivial reducible four-dimensional representation of the group $\A_5$.
There exists a unique $\A_5$-invariant sextic surface $\BB$ with $65$ nodes in $\P^3$,
which is actually the largest possible number of nodes for a sextic surface,
see~\cite{JR}.
This surface was discovered by Barth in~\cite{Barth} and is usually called the \emph{Barth sextic}.
Our second result is

\begin{theorem}
\label{theorem:Barth}
The double cover of $\P^3$ branched over the surface $\BB$ is a rational $\A_5$-Fano threefold
that is $\A_5$-birationally superrigid.
\end{theorem}

Theorems~\ref{theorem:Burkhardt} and~\ref{theorem:Barth}
give two embeddings of the group $\A_5$ into the \emph{space Cremona group} that are not conjugate to each other.
Moreover, these embeddings are also not conjugate to any of the four embeddings described in
\cite[Remark~1.2.1]{CheltsovShramov},
\cite[Example~1.3.9]{CheltsovShramov},
\cite[Theorem~1.4.1]{CheltsovShramov}, and~\cite[Theorem~4.2]{CheltsovPrzyjalkowskiShramov}.

\medskip
\textbf{Notation and conventions.}
All varieties are defined over the field $\C$ of complex numbers.
By $\Cl(X)$ we denote the group of linear equivalence classes of Weil divisors
on a variety~$X$.
If a variety $V$ is acted on by a finite group $G$, and $Z$ is a subvariety of $V$,
we will sometimes abuse terminology and refer to the union of the images $g(Z)$, $g\in G$, as the \emph{$G$-orbit} of $Z$.
By a node we always mean an
isolated singularity that is locally isomorphic to the singularity of a cone over a smooth quadric of
the appropriate dimension.

\medskip
\textbf{Acknowledgements.}
We are grateful to A.\,Kuznetsov, D.\,Pasechnik, V.\,Popov, Yu.\,Prokhorov, and L.\,Rybnikov for useful discussions.
The study has been funded by the Russian Academic Excellence Project ``5-100''.
V.\,Przyjalkowski was %also supported by the grants RFFI 15-01-02158, RFFI 15-01-02164, RFFI 14-01-00160, RFFI 15-51-50045, MK-6019.2016.1.
 partially supported by Laboratory of Mirror Symmetry NRU HSE, RF Government grant, ag. № 14.641.31.0001.
C.\,Shramov was also supported by the grants
RFFI 15-01-02158, RFFI 15-01-02164, RFFI 14-01-00160, and by Dynasty foundation.

\section{Burkhardt quartic}
\label{section:Burkhardt}

In this section we prove Theorem~\ref{theorem:Burkhardt}.
Recall that the Burkhardt quartic $X$ is given in the projective space
$\P^5$ with homogeneous coordinates $x_0,\ldots,x_5$
by equations
$$
\sigma_1(x_0,\ldots,x_5)=\sigma_4(x_0,\ldots,x_5)=0,
$$
where $\sigma_k(x_0,\ldots,x_5)$ is the $k$-th elementary symmetric polynomial in $x_0,\ldots,x_5$.
The quartic~$X$ is rational,
and its automorphism group is isomorphic to $\PSp$.
A subgroup~\mbox{$\SS_6\subset\PSp$}
that acts by permutations of the homogeneous coordinates~\mbox{$x_0,\ldots,x_5$}
preserves~$X$. The Burkhardt quartic
has~$45$ singularities that form two $\SS_6$-orbits:
one is the $\SS_6$-orbit of length $30$ of the point
$$
[1:1:\omega:\omega:\omega^2:\omega^2],
$$
and the other is the $\SS_6$-orbit of length $15$ of the point
$$
[1:-1:0:0:0:0].
$$

\begin{remark}[{cf.~\cite[p.~26]{Atlas}}]
\label{remark:PSp-subgroups}
Up to conjugation, the group $\PSp$ contains a unique subgroup isomorphic
to $\SS_6$, and two subgroups isomorphic to $\A_5$. The latter agree with two non-conjugate
embeddings of $\A_5$ to $\SS_6\subset\PSp$.
\end{remark}

\begin{lemma}\label{lemma:Burkhardt-S6}
One has $\rkCl(X)^{\SS_6}=1$.
\end{lemma}
\begin{proof}
It is well-known that the quotient $\P^5/\SS_6$ is isomorphic to
the weighted projective space~\mbox{$\P(1,2,3,4,5,6)$}, so that
$$
\P^4/\SS_6\cong\P(2,3,4,5,6)
$$
and $X/\SS_6\cong\P(2,3,5,6)$. In particular, one has $\rkCl(X)^{\SS_6}=\rkCl(X/\SS_6)=1$.
\end{proof}

\begin{corollary}\label{corollary:Burkhardt-PSp}
One has $\rkCl(X)^{\PSp}=1$.
\end{corollary}

If one ignores the action of the automorphism group
of $X$, the following result is known.

\begin{lemma}[{\cite[Theorem~1.1(iii)]{Kaloghiros}}]
\label{lemma:Anya}
One has $\rkCl(X)=16$.
\end{lemma}

We are going to find the ranks of the invariant parts of $\Cl(X)$ with respect to various groups.
Consider the vector space
$$
\Cl(X)_{\C}=\Cl(X)\otimes\C
$$
as a representation of the group
$\PSp$. Then $\Cl(X)_{\C}$ contains a trivial subrepresentation~$\mathbb{K}$ corresponding
to the canonical class, and one has $\Cl(X)_{\C}\cong\mathbb{K}\oplus\WW$.
We know from Lemma~\ref{lemma:Anya} that
$\WW$ is a $15$-dimensional representation of $\PSp$.
Obviously, $\WW$ is defined over the field~$\Q$.

\begin{lemma}\label{lemma:Cl-irreducible-action}
The complex $\PSp$-representation $\WW$ is irreducible.
\end{lemma}
\begin{proof}
Recall that all non-trivial irreducible representations of $\PSp$
of dimension less than $15$ are two $5$-dimensional representations $\chi_2$ and $\chi_3$
that are not defined over $\Q$, one $6$-dimensional representation $\chi_4$,
and two $10$-dimensional representations $\chi_5$ and $\chi_6$
that are not defined over $\Q$ (see~\cite[p.~27]{Atlas}).

Suppose that $\WW$ is reducible. Note that $\WW$ does not have trivial
subrepresentations by Corollary~\ref{corollary:Burkhardt-PSp}.
Therefore, $\WW$ splits as a sum of representations of $\PSp$
isomorphic to $\chi_2$, $\chi_3$, $\chi_4$, $\chi_5$, or $\chi_6$.
However, keeping in mind that $\dim(\WW)=15$ we see that there are
no summands isomorphic to $\chi_4$. Moreover, since $\chi_5$ and $\chi_6$ are not
defined over $\Q$, they either appear as summands in $\WW$ simultaneously, or do not
appear at all.
The same holds for $\chi_2$ and $\chi_3$. An obtained contradiction
completes the proof of the lemma.
\end{proof}

By Lemma~\ref{lemma:Cl-irreducible-action}, the $\PSp$-representation $\WW$
is isomorphic to one of the two \mbox{$15$-di}\-men\-sional irreducible representations of $\PSp$ (see~\cite[p.~27]{Atlas}).
In the notation of~\mbox{\cite[p.~27]{Atlas}} these are~$\chi_7$ and~$\chi_8$.
In fact, we have $\WW\cong\chi_7$. This follows from Lemma~\ref{lemma:Burkhardt-S6} and (the first part of) the following result.

\begin{lemma}\label{lemma:splittings}
The following assertions hold:
\begin{itemize}
\item[(i)] the $\SS_6$-representation $\chi_7\vert_{\SS_6}$
does not contain trivial subrepresentations, while
the $\SS_6$-rep\-re\-sen\-tation~\mbox{$\chi_8\vert_{\SS_6}$} does;

\item[(ii)] for one of the two non-conjugate embeddings
of $\A_5$ to $\PSp$, the $\A_5$-rep\-re\-sen\-tation~\mbox{$\chi_7\vert_{\A_5}$}
does not contain trivial subrepresentations, while for the other
embedding it contains a unique trivial subrepresentation.
\end{itemize}
\end{lemma}

\begin{proof}
Both assertions are obtained by direct computations. We used the GAP software~\cite{GAP} to perform them.
\end{proof}

Note that Lemma~\ref{lemma:splittings}(ii) implies that the invariant part of $\Cl(X)$ with respect to
one of the two non-conjugate actions of $\A_5$ on $X$ has rank~$1$.

Later we will need the following elementary result.

\begin{lemma}
\label{lemma:Cl-Q}
Let $Y$ be a normal variety acted on by a finite group $G$.
Suppose that there exist Weil divisors $\Pi_1,\ldots,\Pi_r$ on $Y$ such that they generate the $\mathbb{Q}$-vector space
$$
\Cl(Y)_{\Q}=\Cl(Y)\otimes\Q,
$$
and $\Pi_1,\ldots,\Pi_r$ form one $G$-orbit. Then the $\mathbb{Q}$-vector space
$$
\Cl(Y)_{\Q}^G=\Cl(Y)^G\otimes\Q
$$
is one-dimensional.
\end{lemma}

\begin{proof}
Let $D$ be a $G$-invariant Weil divisor on $Y$.
By assumption, we have
$$
D\sim_{\mathbb{Q}} \sum_{i=1}^{r}a_i\Pi_i
$$
for some rational numbers $a_1,\ldots,a_r$. Put
$\mathcal{P}=\Pi_1+\ldots+\Pi_r$. Then
$$
|G|D\sim\sum_{g\in G}g(D)\sim_{\mathbb{Q}}\sum_{i=1}^{r}a_{i}\sum_{g\in G}g(\Pi_{i})=\sum_{i=1}^{r}a_{i}\frac{|G|}{r}\mathcal{P}=\frac{|G|}{r}\Big(\sum_{i=1}^{r}a_{i}\Big)\mathcal{P}.
$$
In particular, we see that the $\Q$-vector space $\Cl(Y)_{\Q}^G$ is generated by $\mathcal{P}$.
\end{proof}

\begin{corollary}
\label{corollary:Cl-Q}
Let $Y$ be a nodal Fano threefold acted on by a finite group $G$,
and let $i_{Y}$ be the largest positive integer such that $-K_{Y}\sim i_{Y}H$,
where $H$ is an ample Cartier divisor on $Y$.
Suppose that there exist Weil divisors $\Pi_1,\ldots,\Pi_r$ on $Y$ such that they generate the $\mathbb{Q}$-vector space
$$
\Cl(Y)_{\Q}=\Cl(Y)\otimes\Q,
$$
and $\Pi_1,\ldots,\Pi_r$ form one $G$-orbit. Then $\Cl(Y)^G=\mathbb{Z}\cdot H$.
\end{corollary}

\begin{proof}
By Lemma~\ref{lemma:Cl-Q}, the $\mathbb{Q}$-vector space $\Cl(Y)_{\Q}^G$ is one-dimensional.
Since any singular point~$O$ of $Y$ is a node,
we know that any Weil divisor that is $\Q$-Cartier in a neighborhood of~$O$ is actually Cartier in a neighborhood of~$O$.
In particular, every $G$-invariant Weil divisor on $Y$ is a Cartier divisor.
Since the Picard group of $Y$ has no torsion (this holds even for Fano varieties with log terminal singularities,
see e.g. \cite[Proposition~2.1.2]{IsPr99}),
this implies that $\rkCl(Y)^{G}=1$, and the assertion follows.
\end{proof}

Now we are ready to start proving Theorem~\ref{theorem:Burkhardt}.
Recall that a subgroup $\A_5$ in $\SS_6$ that fixes one of the homogeneous
coordinates~\mbox{$x_0,\ldots,x_5$} is called a \emph{standard} subgroup $\A_5$,
and a subgroup $\A_5$ of $\SS_6$ that does not fix any of these homogeneous coordinates
is called a \emph{non-standard} subgroup~$\A_5$.
By Remark~\ref{remark:PSp-subgroups} subgroups of
these two kinds represent two conjugacy classes of subgroups isomorphic to~$\A_5$ in~$\PSp$.

Let $G$ be a standard subgroup $\A_5$ in $\SS_6$ that fixes the homogeneous coordinate $x_5$.
We are going to prove the following result.

\begin{proposition}
\label{proposition:Burkhardt-class-group}
The group $\Cl(X)^G$ is generated by $-K_{X}$.
\end{proposition}

As a consequence of Proposition~\ref{proposition:Burkhardt-class-group}
we can find the rank of the Weil divisor class group invariant
under any given group that contains a standard subgroup $\A_5$.
In particular, for a subgroup~\mbox{$\A_6\subset\SS_6$} we get

\begin{corollary}\label{corollary:Burkhardt-A6}
One has $\rkCl(X)^{\A_6}=1$.
\end{corollary}

\begin{remark}
The assertion of Corollary~\ref{corollary:Burkhardt-A6} was used in the proof of~\cite[Theorem~1.20]{ChSh09b}.
Moreover, the authors of \cite{ChSh09b} gave a brief sketch of a proof of this assertion,
but this proof was actually incorrect. Indeed, contrary to what was
claimed in the proof of~\cite[Theorem~1.20]{ChSh09b},
the quotient of $X$ by a subgroup $\SS_4\subset\A_6$ is not
isomorphic to the weighted projective space $\P(1,2,2,3)$.
Our proof of Corollary~\ref{corollary:Burkhardt-A6} fixes this gap and thus recovers the proof of~\mbox{\cite[Theorem~1.20]{ChSh09b}}.
\end{remark}

\begin{remark}
\label{remark:non-standard}
Let $G^\prime$ be a non-standard subgroup $\A_5\subset\SS_6$. Then
Proposition~\ref{proposition:Burkhardt-class-group} and Lemma~\ref{lemma:splittings}(ii) imply that $\rkCl(X)^{G^\prime}=2$.
\end{remark}

Now we derive Theorem~\ref{theorem:Burkhardt} from Proposition~\ref{proposition:Burkhardt-class-group}.

\begin{proof}[Proof of Theorem~\ref{theorem:Burkhardt}]
The proof is standard, see \cite[Theorem~1.20]{ChSh09b} and \cite{Mella}, but we include it for the reader's convenience.

Suppose that $X$ is not $G$-birationally superrigid.
Since the divisor $-K_{X}$ generates the group~\mbox{$\Cl(X)^G$} by Proposition~\ref{proposition:Burkhardt-class-group}, it follows from \cite[Corollary~3.3.3]{CheltsovShramov}
that there exist a positive integer~$n$ and
a $G$-invariant linear subsystem $\mathcal{M}$ of the linear system $|-nK_X|$
such that~$\mathcal{M}$ does not have fixed components, and the~log pair $(X,\frac{1}{n}\mathcal{M})$ is not canonical.
Choose two general surfaces~$M_{1}$ and~$M_{2}$ in the linear system $\mathcal{M}$,
and denote by $H$ a general hyperplane section of~$X$. Then~\mbox{$H\sim -K_X$}.

Suppose that there is an irreducible curve $C\subset X$ such that the~log pair $(X,\frac{1}{n}\mathcal{M})$ is not canonical along $C$.
Denote by $Z$ the $G$-orbit of the curve~$C$.
Put \mbox{$d=H\cdot Z$} and~\mbox{$m=\mathrm{mult}_{C}(M_{1})=\mathrm{mult}_{C}(M_{2})$}.
Then $m>n$, so that
$$
4n^2=M_{1}\cdot M_{2}\cdot H\geqslant dm^2>dn^2,
$$
which gives $d\leqslant 3$.
In particular, one has $Z=C$. Since $\P^4$ does not contain $G$-invariant lines and planes,
we see that $Z$ is a twisted cubic. Moreover, the curve $Z$ is contained in the smooth locus of $X$,
because the possible lengths of $G$-orbits in $Z$ are $12$, $20$, $30$, $60$,
and there are no $G$-orbits of such lengths consisting of singular points of $X$.

Let $f\colon W\to X$ be the blow up along the curve $Z$.
Denote by $F$ the $f$-ex\-cep\-tional divisor.
Denote by $\widetilde{M}_1$ and $\widetilde{M}_2$ the proper transforms of the surfaces $M_{1}$ and $M_{2}$
on the threefold $W$, respectively.
We then get
\begin{equation*}
\widetilde{M}_1\sim \widetilde{M}_2\sim f^{*}(nH)-mF.
\end{equation*}
Moreover, the divisor $f^{*}(2H)-F$ is nef,
because $Z$ is a scheme-theoretic intersection of quadrics.
Keeping in mind that $F^{3}=-K_{X}\cdot Z-2=1$ and $H\cdot Z=3$, we obtain
$$
0\leqslant\Big(f^{*}(2H)-F\Big)\cdot\widetilde{M}_1\cdot\widetilde{M}_2=\Big(f^{*}(2H)-F\Big)\Big(f^{*}(nH)-mF\Big)^2=8n^2-6nm-5m^2<0
$$
because $m>n$. This is a contradiction.

Thus, the log pair $(X,\lambda\mathcal{M})$ is canonical outside of finitely many points of $X$.
Take any point~\mbox{$P\in X$} such that the~singularities of the~log pair $(X,\lambda\mathcal{M})$ are not canonical at the point~$P$.
Suppose that $P$ is a smooth point of $X$, and let $H_P$ be a general hyperplane section of $X$ passing through $P$. Then
$$
4m=M_1\cdot M_2\cdot H_P\geqslant\mathrm{mult}_P\Big(M_1\cdot M_2\Big)>4n^2
$$
by~\cite[Corollary~3.4]{Co00}. The obtained contradiction shows that $X$ is singular at~$P$.

Let $\Sigma$ be the~$G$-orbit of the~point $P$.
Then there is a subset $\Gamma\subset\Sigma$ such that~\mbox{$|\Gamma|=4$},
and the~set $\Gamma$ is not contained in any plane in $\mathbb{P}^{4}$.
Let $g\colon U\to X$ be a~blow up of $\Gamma$,
and let $E_{1}$, $E_2$, $E_3$, and $E_{4}$ be exceptional divisors of $g$.
Denote by $\overline{M}_1$ and $\overline{M}_2$ the proper transforms of the surfaces $M_{1}$ and $M_{2}$ on the threefold $U$, respectively.
We then get
$$
\overline{M}_1\sim\overline{M}_2\sim g^{*}(nH)-\delta \sum_{i=1}^{4}E_i.
$$
for some positive integer $\delta$.
Moreover, it follows from \cite[Theorem~3.10]{Co00} that $\delta>n$.
On the other hand, the divisor $g^*(2H)-E_1-E_2-E_3-E_4$ is nef,
because the points of $\Gamma$ are not coplanar.
In particular, we have
$$
0\leqslant\Big(g^{*}(2H)-\sum_{i=1}^{4}E_i\Big)\cdot\overline{M}_1\cdot\overline{M}_2=\Big(f^{*}(2H)-\sum_{i=1}^{4}E_i\Big)\Big(f^{*}(nH)-\delta \sum_{i=1}^{4}E_i\Big)^2=8n^2-8\delta^2.
$$
This is impossible, since $\delta>n$.
\end{proof}

In the rest of this section we give a proof of Proposition~\ref{proposition:Burkhardt-class-group}.
Recall from \cite[\S5.2.1]{Hunt} that the Burkhardt quartic $X$
contains forty \emph{$j$-planes}, that are planes passing through nine
singular points of $X$. Let us describe them. For any triple of indices
$0\le i_1<i_2<i_3\le 5$, we denote by~$\Pi_{i_1 i_2 i_3}^+$
the plane given in $\P^5$ by equations
$$
x_{i_2}=\omega x_{i_1},\quad
x_{i_3}=\omega^2 x_{i_1},\quad
\sigma_1(x_0,\ldots,x_5)=0,
$$
where $\omega$ is a primitive cubic root of $1$,
and we denote by $\Pi_{i_1 i_2 i_3}^-$
the plane given in $\P^5$ by equations
$$
x_{i_3}=\omega x_{i_1},\quad
x_{i_2}=\omega^2 x_{i_1},\quad
\sigma_1(x_0,\ldots,x_5)=0.
$$

The $40$ planes $\Pi_{i_1 i_2 i_3}^{\pm}$
form one $\SS_6$-orbit.
On the other hand, these $40$ planes split into two $G$-orbits of length $20$,
one containing the planes $\Pi_{i_1 i_2 i_3}^{\pm}$ for $0\le i_1<i_2<i_3\le 4$,
and the other containing $\Pi_{i_1 i_2 5}^{\pm}$ for $0\le i_1<i_2\le 4$.

Consider the following intersection form on the
lattice $\Cl(X)$. Choose a general (smooth) hyperplane section
$H$ of $X$; in particular, we assume that $H$ does not pass
through the singular points of~$X$.
Given two Weil divisors $D_1$ and $D_2$ on $X$,
we restrict them to $H$ (which makes sense since in
appropriately chosen neighborhoods
of the intersection of their supports with $H$ they are actually
Cartier divisors), and define $D_1\bullet D_2$ as the intersection
of the resulting curves on~$H$.

\begin{lemma}\label{lemma:intersection-Pi}
Choose two triples of indices
$$
0\le i_1<i_2<i_3\le 5,\quad
0\le j_1<j_2<j_3\le 5.
$$
Let $c$ be the cardinality
of the set $\{i_1,i_2,i_3\}\cap\{j_1,j_2,j_3\}$.
If $c=2$, put $\delta=1$ provided that one can choose indices
$1\le a<b\le 3$ and $1\le a'<b'\le 3$ such that $i_a=j_{a'}$,
$i_b=j_{b'}$, and~\mbox{$b-a=b'-a'$}; otherwise put $\delta=0$.
The following assertions hold:
\begin{itemize}
\item[(o)] if $c=0$, then
$\Pi_{i_1 i_2 i_3}^+\bullet \Pi_{j_1 j_2 j_3}^+=\Pi_{i_1 i_2 i_3}^+\bullet \Pi_{j_1 j_2 j_3}^-=\Pi_{i_1 i_2 i_3}^-\bullet \Pi_{j_1 j_2 j_3}^-=1$;

\item[(i)] if $c=1$, then
$\Pi_{i_1 i_2 i_3}^+\bullet \Pi_{j_1 j_2 j_3}^+=\Pi_{i_1 i_2 i_3}^+\bullet \Pi_{j_1 j_2 j_3}^-=\Pi_{i_1 i_2 i_3}^-\bullet \Pi_{j_1 j_2 j_3}^-=0$;

\item[(ii)] if $c=2$, then
$\Pi_{i_1 i_2 i_3}^+\bullet \Pi_{j_1 j_2 j_3}^+=\Pi_{i_1 i_2 i_3}^-\bullet \Pi_{j_1 j_2 j_3}^-=\delta$,
and $\Pi_{i_1 i_2 i_3}^+\bullet \Pi_{j_1 j_2 j_3}^-=1-\delta$;

\item[(iii)] one has
$\Pi_{i_1 i_2 i_3}^+\bullet \Pi_{i_1 i_2 i_3}^+=\Pi_{i_1 i_2 i_3}^-\bullet \Pi_{i_1 i_2 i_3}^-=-2$,
and $\Pi_{i_1 i_2 i_3}^+\bullet \Pi_{i_1 i_2 i_3}^-=1$.
\end{itemize}
\end{lemma}

\begin{proof}
The self-intersection number $-2$ corresponds to the self-intersection of
a smooth rational curve on a $K3$ surface.
The cases with intersection number $1$ correspond to pairs of
planes that meet along a line, and the cases with
intersection number $0$ correspond to pairs of
planes that meet at a point.
\end{proof}

\begin{corollary}\label{corollary:16}
The $20\times 20$ matrix of intersection numbers of the planes $\Pi_{i_1 i_2 i_3}^{\pm}$,
where~\mbox{$0\le i_1<i_2<i_3\le 4$}, has rank~$16$.
Similarly, the $20\times 20$ matrix of intersection numbers of the planes $\Pi_{i_1 i_2 5}^{\pm}$, where $0\le i_1<i_2\le 4$,
also has rank $16$.
\end{corollary}

\begin{proof}
Straightforward computation.
\end{proof}

Lemma~\ref{lemma:Anya} and Corollary~\ref{corollary:16}
imply the following result.

\begin{corollary}
\label{corollary:Pi-generate}
The classes of the $20$ planes $\Pi_{i_1 i_2 i_3}^{\pm}$,
where $0\le i_1<i_2<i_3\le 4$, generate the $\Q$-vector space
$$
\Cl(X)_{\Q}=\Cl(X)\otimes\Q.
$$
Similarly, the classes of the $20$ planes  $\Pi_{i_1 i_2 5}^{\pm}$, where $0\le i_1<i_2\le 4$,
also generate the $\Q$-vector space $\Cl(X)_{\Q}$.
\end{corollary}

By Corollary~\ref{corollary:Cl-Q}, the assertion of Proposition~\ref{proposition:Burkhardt-class-group} follows
from Corollary~\ref{corollary:Pi-generate}. This completes the proof of Theorem~\ref{theorem:Burkhardt}.

\section{Barth sextic double solid}
\label{section:Barth}

Let $\II$ be the trivial representation of the group $\A_5$, and let
$\VV$ be one of its two three-dimensional irreducible representations
(see e.g.~\cite[p.~2]{Atlas}). Put $\P^3=\P(\II\oplus\VV)$.
By \cite{Barth}, there exists a unique $\A_5$-invariant sextic surface $\BB$ in $\P^3$ with $65$ isolated singular points.
Moreover, the surface $\BB$ is given in appropriate homogeneous coordinates $x_0,x_1,x_2,x_3$ by equation
\begin{equation}\label{eq:Barth}
4(\tau^2x_0^2-x_1^2)(\tau^2x_1^2-x_2^2)
(\tau^2x_2^2-x_0^2)-(1+2\tau)x_3^2(x_0^2+x_1^2+x_2^2-x_3^{2})^2=0,
\end{equation}
where $\tau=\frac {1+\sqrt{5}}{2}$, and it has only nodes as singularities.

Recall from \cite[\S1]{Barth} that the group $\A_5$ acting on $\P^3$
so that the sextic $\BB$ is $\A_5$-invariant can be thought of as
the group of rotations of an icosahedron with $12$ vertices
$$
[\pm \tau:\pm 1:0:1], \quad [0:\pm \tau:\pm 1:1], \quad [\pm 1:0:\pm \tau:1].
$$
In particular, the group $\A_5$ contains the transformation
\begin{equation}\label{eq:N}
x_0\mapsto -x_0, \quad x_1\mapsto -x_1, \quad x_2\mapsto x_2, \quad x_3\mapsto x_3
\end{equation}
of order $2$,
the transformation
\begin{equation}\label{eq:R}
x_0\mapsto x_1\mapsto x_2\mapsto x_0, \quad x_3\mapsto x_3
\end{equation}
of order $3$, and the transformation
\begin{equation}\label{eq:M}
(x_0,x_1,x_2)\mapsto (x_0,x_1,x_2) \mathrm{M}^T, \quad x_3\mapsto x_3,
\end{equation}
where $\mathrm{M}$ is the matrix
\begin{multline*}
\left(
\begin{array}{ccc}
\frac{\tau}{\sqrt{\tau+2}} & -\frac{1}{\sqrt{\tau+2}} & 0\\
\frac{1}{\sqrt{\tau+2}} & \frac{\tau}{\sqrt{\tau+2}} & 0\\
0 & 0 & 1
\end{array}
\right)
\left(
\begin{array}{ccc}
1 & 0 & 0\\
0 & \cos\left(\frac{2\pi}{5}\right) & -\sin\left(\frac{2\pi}{5}\right)\\
0 & \sin\left(\frac{2\pi}{5}\right) & \cos\left(\frac{2\pi}{5}\right)
\end{array}
\right)
\left(
\begin{array}{ccc}
\frac{\tau}{\sqrt{\tau+2}} & -\frac{1}{\sqrt{\tau+2}} & 0\\
\frac{1}{\sqrt{\tau+2}} & \frac{\tau}{\sqrt{\tau+2}} & 0\\
0 & 0 & 1
\end{array}
\right)^{-1}=\\
=\left(
   \begin{array}{ccc}
     \frac{\sqrt{5}(3+\sqrt{5})}{2(5+\sqrt{5})} & \frac{\sqrt{5}}{5+\sqrt{5}} & \frac{1}{2} \\
     \frac{\sqrt{5}}{5+\sqrt{5}} & \frac{1}{2} & \frac{-\sqrt{5}-1}{4} \\
     -\frac{1}{2} & \frac{\sqrt{5}+1}{4} & \frac{\sqrt{5}-1}{4} \\
   \end{array}
 \right)
=\frac{1}{2}\left(
   \begin{array}{ccc}
      \tau & \tau-1 & 1\\
      \tau-1 & 1 & -\tau\\
      -1 & \tau & \tau-1\\
   \end{array}
\right).
\end{multline*}
The latter is a transformation
of order $5$ which corresponds to a rotation around the axis
through the vertices $[\tau:1:0:1]$ and $[-\tau:-1:0:1]$ of the icosahedron by the angle~$2\pi/5$.

Let $\Sigma_{15}$ be the $\A_5$-orbit of the point $[1:0:0:0]$,
let $\Sigma_{30}$ be the $\A_5$-orbit of the point~\mbox{$[1:0:0:1]$},
and let $\Sigma_{20}$ be the $\A_5$-orbit of the point $[1:1:1:1]$.
Then one has~\mbox{$|\Sigma_k|=k$},
and one can check that
the sextic surface $\BB$ is singular at the points of these
three $\A_5$-orbits.
Moreover, one has $\mathrm{Sing}(\BB)=\Sigma_{15}\cup\Sigma_{20}\cup\Sigma_{30}$,
see \cite[\S1]{Barth}.

\begin{remark}\label{remark:square}
Restricting the left hand side of \eqref{eq:Barth} to the plane $x_3=x_0+x_1+x_2$
we get an equation
$$
-4(5\tau+3)\left((\tau-2)(x_0x_1^2+x_1x_2^2+x_2x_0^2)+(\tau-3)x_0x_1x_2-(x_0^2x_1+x_1^2x_2+x_2^2x_0)\right)^2=0.
$$
Similarly, restricting the left hand side of \eqref{eq:Barth}
to the plane $x_3=x_0-x_1-x_2$ we get an equation
$$
-4(5\tau+3)\left(
(2-\tau)(x_0x_1^2-x_1x_2^2-x_2x_0^2)+(3-\tau)x_0x_1x_2-(x_0^2x_1+x_1^2x_2-x_2^2x_0)
\right)^2=0.
$$
\end{remark}

Define the plane $\Xi_{(v_0,v_1,v_2)}$ in $\P^3$ by equation $x_3=v_0x_0+v_1x_1+v_2x_2$,
where $v=(v_0,v_1,v_2)$ is one of the
following collections of coefficients:
\begin{equation}
\label{equation:v}
\aligned
&(1,1,1),\ (1,1,-1),\ (1,-1,1),\ (-1,1,1),\\
&(1,-1,-1),\ (-1,1,-1),\ (-1,-1,1),\ (-1,-1,-1),\\
&(\tau-1, \tau,0),\ (1-\tau,\tau,0),\ (1-\tau, -\tau, 0),\ (\tau-1,-\tau,0), (\tau, 0, 1-\tau),\ (\tau,0,\tau-1),\\
&(-\tau,0,\tau-1),\  (-\tau,0,1-\tau),\ (0,\tau-1,\tau),\ (0,\tau-1,-\tau),\ (0,1-\tau,-\tau),\ (0,1-\tau,\tau).
\endaligned
\end{equation}
There are $20$ planes like this, and they form a single $\A_5$-orbit.
Similarly, define the plane~$\Theta_{(u_0,u_1,u_2)}$ in $\P^3$ by equation $u_0x_0+u_1x_1+u_2x_2=0$,
where $u=(u_0,u_1,u_2)$ is one of the
following collections of coefficients:
$$
(\tau,1,0),\ (\tau,-1,0),\ (0,\tau,1),\ (0,\tau,-1),\ (1,0,\tau),\ (-1, 0, \tau).
$$
There are $6$ planes like this, and they form
a single $\A_5$-orbit.

\begin{lemma}\label{lemma:restriction-double-cubic}
A restriction of the sextic $\BB$ to each of the planes
$\Xi_v$ is a smooth cubic curve taken with multiplicity~$2$.
A restriction of the sextic $\BB$ to each of the planes
$\Theta_u$ is a union of a line taken with multiplicity~$2$ and an irreducible conic taken with multiplicity~$2$.
\end{lemma}

\begin{proof}
It is enough to check the assertion for one of the planes $\Xi_v$ and one of the planes~$\Theta_u$.
The restriction of $\BB$ to the plane $\Xi_{1,1,1}$
is given by equation
$$
\left((\tau-2)(x_0x_1^2+x_1x_2^2+x_2x_0^2)+(\tau-3)x_0x_1x_2-(x_0^2x_1+x_1^2x_2+x_2^2x_0)\right)^2=0,
$$
see Remark~\ref{remark:square}.
Similarly, the restriction of $\BB$ to the plane $\Theta_{(-1,0,\tau)}$
is given by equation
$$
x_3^2\left(x_1^2+(1+\tau^2)x_2^2-x_3^{2}\right)^2=0.
$$
\end{proof}

Denote by $\Upsilon$ the plane in $\mathbb{P}^3$ that is given by $x_3=0$,
so that $\Upsilon\cong\P(\VV)$.
Recall that
for every~\mbox{$k\in\{6,10,15\}$} there is a unique $\A_5$-orbit $\Omega_k$ of length $k$ in $\Upsilon$,
and there is a unique $\A_5$-invariant curve $\LL_k$ in $\Upsilon$ that is a union of
$k$ lines, see e.g.~\cite[Lemma~5.3.1(i),(ii)]{CheltsovShramov}.

\begin{lemma}
\label{lemma:Barth-lines}
Let $\ell$ be a line in $\Upsilon$ that is not an irreducible component of $\mathcal{L}_6$.
Suppose that~\mbox{$(\ell\cdot\mathcal{L}_6)_P\geqslant 2$}
for every point $P\in\ell\cap\mathcal{L}_6$. Then $\ell$ is an irreducible component of $\mathcal{L}_{10}$.
\end{lemma}

\begin{proof}
By~\cite[Theorem~6.1.2(i)]{CheltsovShramov}
the singular points of the curve $\LL_6$ are the points of $\Omega_{15}$;
the multiplicity of $\LL_6$ at each of these points equals~$2$.
Therefore, the line $\ell$ must contain three points of $\Omega_{15}$.
On the other hand, all lines passing through pairs of points of
$\Omega_{15}$ are irreducible
components of the curves $\LL_6$, or $\LL_{10}$, or $\LL_{15}$;
this follows from polarity (see \cite[Remark~5.3.2]{CheltsovShramov})
and the fact that the points of pairwise intersections
of the irreducible components of $\LL_{15}$ are the points of the $\A_5$-orbits
$\Omega_6$, $\Omega_{10}$, and $\Omega_{15}$, see \cite[Theorem~6.1.2(xvi)]{CheltsovShramov}.
This implies that~$\ell$ is an irreducible
component of either $\LL_6$, or $\LL_{10}$, or $\LL_{15}$.
However, the first of these cases does not occur by assumption,
and the third is excluded by \cite[Theorem~6.1.2(xiv)]{CheltsovShramov}.
\end{proof}

Note that the intersection $\Xi_v\cap\Upsilon$ is an irreducible component of $\mathcal{L}_{10}$.
Indeed, this intersection is a line in $\Upsilon$ whose $\A_5$-orbit has length $k$ that divides~$20$,
and moreover $k<20$ because $\Xi_{(1,1,1)}$ and $\Xi_{(-1,-1,-1)}$ intersect $\Upsilon$ by the same line.
Similarly, the intersection~\mbox{$\Theta_u\cap\Upsilon$} is an irreducible component of $\mathcal{L}_{6}$,
because the latter is the only $\A_5$-orbit in $\Upsilon$ that consists of at most $6$ lines.

\begin{proposition}
\label{proposition:Barth-planes}
Suppose that $\Pi$ is a plane in $\P^3$ such that the restriction
$\BB\vert_{\Pi}$ is a cubic curve taken with multiplicity~$2$.
Then $\Pi$ is one of the planes $\Xi_v$ or $\Theta_u$.
\end{proposition}

\begin{proof}
Observe that $\Pi\ne\Upsilon$, because $\BB\vert_{\Upsilon}=\mathcal{L}_{6}$ is a reduced curve.
So we put $\ell=\Pi\cap\Upsilon$.
By Lemma~\ref{lemma:Barth-lines}, the line $\ell$ is
an irreducible component of either $\mathcal{L}_{6}$ or $\mathcal{L}_{10}$.
Since $\A_5$ permutes transitively the irreducible components of each of the curves $\mathcal{L}_{6}$ and $\mathcal{L}_{10}$,
we may assume that~$\ell$ is given either by
$\tau x_0+x_1=x_3=0$, or by $x_0+x_1+x_2=x_3=0$.

Suppose that $\ell$ is given by $\tau x_0+x_1=x_3=0$.
Then $\Pi$ is given by the equation~\mbox{$\tau x_0+x_1=\lambda x_3$}
for some $\lambda\in\mathbb{C}$. If $\lambda=0$, then $\Pi=\Theta_{(\tau,1,0)}$.
If $\lambda\ne 0$, then the  restriction $\BB\vert_{\Pi}$ is given by
\begin{multline*}
0=4(\tau^2x_0^2-x_1^2)(\tau^2x_1^2-x_2^2)
(\tau^2x_2^2-x_0^2)\\
-\frac{1+2\tau}{\lambda^2}(\tau x_0+x_1)^2\left(x_0^2+x_1^2+x_2^2-\frac{1}{\lambda^2}(\tau x_0+x_1)^2\right)^2.
\end{multline*}
This (possibly non-reduced) sextic curve contains the line $\tau x_0+x_1=0$ with multiplicity~$1$,
so that it cannot be a double cubic.

Now we suppose that $\ell$ is given by $x_0+x_1+x_2=x_3=0$.
Then the plane $\Pi$ is given by the equation $x_0+x_1+x_2=\lambda x_3$
for some $\lambda\in\mathbb{C}$. If $\lambda=1$, then $\Pi=\Xi_{(1,1,1)}$.
If $\lambda=-1$, then~\mbox{$\Pi=\Xi_{(-1,-1,-1)}$}.

Suppose that $\lambda\ne 0$. Put $\mu=\frac{1}{\lambda}$. Then $\mu\ne 0$, and the restriction $\BB\vert_{\Pi}$ is given by the equation~\mbox{$f(x_0,x_1,x_2)=0$}, where
\begin{multline*}
f(x_0,x_1,x_2)=4(\tau^2x_0^2-x_1^2)(\tau^2x_1^2-x_2^2)(\tau^2x_2^2-x_0^2)\\
-(1+2\tau)\mu^2(x_0+x_1+x_2)^2\left(x_0^2+x_1^2+x_2^2-\mu^2(x_0+x_1+x_2)^2\right)^2.
\end{multline*}
We have to show that the polynomial $f(x_0,x_1,x_2)$ is not a square of a cubic polynomial unless~\mbox{$\mu=\pm 1$}.
To show this it is enough to prove the same assertion for the polynomial
\begin{multline*}
f(1,x_1,-x_1)=-(8\tau+4)x_1^6-(4+8\tau)(\mu^2-3)x_1^4\\
+(4+8\tau)(\mu^2-\tau)(\mu^2+\tau-1)x_1^2-(1+2\tau)\mu^2(\mu+1)^2(\mu-1)^2.
\end{multline*}
This follows from the fact that if $\mu\ne\pm 1$, then both the constant term and the leading coefficient
of $f(1,x_1,-x_1)$ are not zero, while the coefficients at $x_1$ and $x_1^3$ are both zero.

Thus, we see that $\lambda=0$, so that $\Pi$ is given by $x_0+x_1+x_2=0$.
Expressing $x_0=-x_1-x_2$, we see that the restriction $\BB\vert_{\Pi}$ is given by
the equation $g(x_1,x_2,x_3)=0$, where
\begin{multline*}
g(x_1,x_2,x_4)=4(\tau^2(x_1+x_2)^2-x_1^2)(\tau^2x_1^2-x_2^2)\left(\tau^2x_2^2-(x_1+x_2)^2\right)\\
-(1+2\tau)x_3^2\left((x_1+x_2)^2+x_1^2+x_2^2-x_3^{2}\right)^2.
\end{multline*}
We have to show that the polynomial $g(x_1,x_2,x_3)$ is not a square of a cubic polynomial.
To show this it is enough to prove the same assertion for the polynomial
$$
g(x_1,-x_1,1)=-(1+2\tau)\big(4x_1^6+4x_1^4-4x_1^2+1\big).
$$
This polynomial is not a square of a cubic polynomial, because both its constant term and the leading coefficient are not zero, while the coefficients at $x_1$ and $x_1^3$ are both zero.
\end{proof}

Let $\pi\colon X\to\P^3$ be a double cover branched over the sextic~$\BB$.
The equation of $X$ can be written in the weighted projective space~\mbox{$\P(1,1,1,1,2)$} with weighted homogeneous
coordinates~\mbox{$x_0,\ldots, x_3$}, and~$w$ as
\begin{equation}\label{eq:Barth-double-solid}
w^2+4l_1l_2l_3l_4l_5l_6-q_3^2=0,
\end{equation}
where
$l_1=\tau x_0-x_1$, $l_2=\tau x_1-x_2$,
$l_3=\tau x_2-x_0$, $l_4=\tau x_0+x_1$,
$l_5=\tau x_1+x_2$, $l_6=\tau x_2+x_0$,
and
$$
q_3=\sqrt{1+2\tau}x_3(x_0^2+x_1^2+x_2^2-x_3^2).
$$

The class group of the threefold $X$ was described by Endrass.

\begin{lemma}[{\cite[Example~3.7]{En99}}]
\label{lemma:Stefan}
One has $\rkCl(X)=14$.
\end{lemma}

\begin{proposition}
\label{proposition:Barth-rational}
The threefold $X$ is rational.
\end{proposition}

\begin{proof}
Making a change of coordinates $w=2yl_1l_2+q_3$
we see that there is a birational map~\mbox{$\varphi\colon X\dasharrow Y$} to a quartic threefold
$Y$ given in the projective space $\P^4$ with homogeneous coordinates
$x_0,\ldots,x_3,y$ by equation
$
y^2l_1l_2+yq_3+l_3l_4l_5l_6=0.
$
The map $\varphi$ is given by the formula
$$
[x_0:x_1:x_2:x_3:w]\mapsto [2l_1l_2x_0:2l_1l_2x_1:2l_1l_2x_2:2l_1l_2x_3:w-q_3].
$$
The inverse birational map $\psi\colon Y\dasharrow X$ is given
by the Stein factorization of the linear projection from
the point $[0:0:0:0:1]$ in $\P^4$, so that
$\psi$ is defined by the formula
$$
[x_0:x_1:x_2:x_3:y]\mapsto [x_0:x_1:x_2:x_3:2l_1l_2y+q_3].
$$
The quartic $Y$ contains a plane $\Pi$ given by equations
$y=l_4=0$.
The projection~\mbox{$\sigma\colon Y\dasharrow \P^1$} from $\Pi$ is given
by $[\lambda:\mu]=[y:l_4]$,
where $\lambda$ and $\mu$ are homogeneous coordinates on $\P^1$.
$$
\xymatrix{
&&Y\ar@{^{(}->}[rr]\ar@/^/@{-->}[dd]^{\psi}\ar@{-->}[lldd]_{\sigma} && \P^4\ar@{-->}[dd]\\
&&&\\
\P^1&&X\ar@/^/@{-->}[uu]^{\varphi}\ar@{->}[rr]^{\pi} && \P^3}
$$

Putting $\lambda=\frac{y}{l_4}$ and $\mu=1$, we see that the general fiber of $\sigma$ (in a scheme sence)
is a cubic surface in the projective space $\P^3_{\mathbb{F}}$ over the field $\mathbb{F}=\mathbb{C}(\lambda)$,
that is given by equation
\begin{equation}
\label{eq:cubic}
\lambda^2l_1l_2l_4+\lambda q_3+l_3l_5l_6=0.
\end{equation}
Here we use $x_0,\ldots,x_3$ also as a homogeneous coordinates on $\P^3_{\mathbb{F}}$.

Restricting the left hand side of equation~\eqref{eq:cubic}
to the plane $\Pi'$ given by
$x_3=x_0+x_1+x_2$
and using Remark~\ref{remark:square}, we see
that the corresponding curve
is given by equation
\begin{equation}
\left( \lambda l_4+\sqrt{2\tau+1}(2\tau-3)l_3\right)
\left(\lambda l_1l_2+\sqrt{2\tau+1}l_5l_6\right)=0.
\end{equation}
Restricting the left hand side of equation~\eqref{eq:cubic}
to the plane $\Pi''$ given by
$x_3=x_0-x_1-x_2$
and using Remark~\ref{remark:square}, we see
that the corresponding curve
is given by equation
\begin{equation}
\left( \lambda l_1+\sqrt{2\tau+1}(2\tau-3)l_6\right)
\left(\lambda l_2l_4+\sqrt{2\tau+1}l_3l_5\right)=0.
\end{equation}
One can check that the lines in $\P^3_{\mathbb{F}}$ given by equations
$$
x_3-x_0-x_1-x_2=%\left(
\lambda l_4+\sqrt{2\tau+1}(2\tau-3)l_3=0
$$
and
$$
x_3-x_0+x_1+x_2=
\lambda l_1+\sqrt{2\tau+1}(2\tau-3)l_6=0
$$
are disjoint.
Since they are contained in the cubic surface~\eqref{eq:cubic},
we see that the cubic surface~\eqref{eq:cubic} is rational over the field~$\mathbb{F}$,
so that both $Y$ and $X$ are rational (over the field~$\C$).
\end{proof}

Now we are going to describe the generators of the group $\Cl(X)$.

The intersection of $X$ with the hypersurface $x_3=x_0+x_1+x_2$ splits as a union of two
surfaces~$\Xi_{(1,1,1)}^{+}$ and~$\Xi_{(1,1,1)}^{-}$ that are given by
$$
\left\{\aligned%
&x_3=x_0+x_1+x_2,\\
&w=C_{\pm}\cdot\Big((\tau-2)(x_0x_1^2+x_1x_2^2+x_2x_0^2)+(\tau-3)x_0x_1x_2-(x_0^2x_1+x_1^2x_2+x_2^2x_0)\Big),
\endaligned
\right.
$$
respectively, where $C_{\pm}=\pm2\sqrt{5\tau+3}$.
The image $\pi(\Xi_{(1,1,1)}^{+})=\pi(\Xi_{(1,1,1)}^{-})$ is the plane in $\mathbb{P}^3$ that is
given by equation~\mbox{$x_3=x_0+x_1+x_2$}, cf. Lemma~\ref{lemma:restriction-double-cubic}.
Since the $\A_5$-orbit of the plane~\mbox{$x_3=x_0+x_1+x_2$} consists of~$20$ planes,
the $\A_5$-orbit of the surface $\Xi_{(1,1,1)}^{+}$ consists of~$20$ surfaces.
Similarly, the $\A_5$-orbit of the surface $\Xi_{(1,1,1)}^{-}$ also consists of $20$ surfaces.
Denote by~$\Xi_{v}^{+}$ the surface in the the $\A_5$-orbit of~$\Xi_{(1,1,1)}^{+}$
such that~\mbox{$\pi(\Xi_v^{+})=\Xi_v$}, where~\mbox{$v=(v_0,v_1,v_2)$} is one of the
collections of coefficients listed in~\eqref{equation:v}.

\begin{remark}
The anticanonical degree of the surfaces $\Xi_v^{\pm}$ equals $1$,
i.e. one has $\Xi_v^{\pm}\cdot K_{X}^2=1$.
By Lemma~\ref{lemma:restriction-double-cubic}, the preimage on $X$
of a plane $\Theta_u$ also splits as a union of two surfaces $\Theta_u^{+}$ and $\Theta_u^{-}$
of anticanonical degree~$1$. It follows from Proposition~\ref{proposition:Barth-planes}
that there are no surfaces of anticanonical degree~$1$ on $X$ except
$\Xi_v^{\pm}$ and~$\Theta_u^{\pm}$.
\end{remark}

Fix a sufficiently general (smooth) $K3$ surface $S$ in the linear system $|-K_{X}|$.
For every two surfaces $\Xi_{v}^{+}$ and $\Xi_{v'}^{+}$, put
$$
\Xi_{v}^{+}\bullet\Xi_{v'}^{+}=\Xi_{v}^{+}\vert_{S}\cdot\Xi_{v'}^{+}\vert_{S}.
$$
Then~\mbox{$\Xi_{v}^{+}\bullet\Xi_{v}^{+}=-2$} by the adjunction formula.
Moreover, if  $v\ne v'$, then either~\mbox{$\Xi_{v}^{+}\bullet\Xi_{v'}^{+}=1$}
or~\mbox{$\Xi_{v}^{+}\bullet\Xi_{v'}^{+}=0$} by construction.
Furthermore, if $v\ne v'$, then~\mbox{$\Xi_{v}^{+}\bullet\Xi_{v'}^{+}=0$} if and only if
the intersection~\mbox{$\Xi_{v}^{+}\cap\Xi_{v'}^{+}$} consists of finitely many points.

Denote by $N$, $R$, and $M$ the transformations~\eqref{eq:N}, ~\eqref{eq:R}, and~\eqref{eq:M},
respectively.

\begin{example}
By Remark~\ref{remark:square}, the surface $\Xi_{(1,1,1)}^+$ is defined in $\P(1,1,1,1,2)$ by equations
$$
\left\{
  \begin{array}{ll}
    x_3=x_0+x_1+x_2, \\
    w=(\tau-2)(x_0x_1^2+x_1x_2^2+x_2x_0^2)+(\tau-3)x_0x_1x_2-(x_0^2x_1+x_1^2x_2+x_2^2x_0),
  \end{array}
\right.
$$
The transformation $M^{3}$ is given by the matrix
$$
\mathrm{M}^{3}=\frac{1}{2}\left(
                     \begin{array}{ccc}
                       1 & \tau & 1-\tau \\
                       \tau & 1-\tau & 1 \\
                       \tau-1 & -1 & -\tau \\
                     \end{array}
                   \right).
$$
Thus the surface $\Xi_{(1,1,-1)}^+$ is defined by equations
$$
\left\{
  \begin{array}{ll}
    x_3=x_0+x_1-x_2, \\
    w=(\tau-2)(x_0x_1^2+x_1x_2^2-x_2x_0^2)-(\tau-3)x_0x_1x_2-(x_0^2x_1-x_1^2x_2+x_2^2x_0).
  \end{array}
\right.
$$
Therefore, the intersection $\Xi_{(1,1,1)}^+\cap \Xi_{(1,-1,-1)}^+$ is
the line in $\Xi_{(1,1,1)}^+$ that is cut out by
an equation~\mbox{$x_2=0$},
so that~\mbox{$\Xi_{(1,1,1)}^+\bullet\Xi_{(1,-1,-1)}^+=1$}.
\end{example}

\begin{example}
The transformation $RN$ is given by the matrix
$$
\mathrm{RN}=\left(
                     \begin{array}{ccc}
                       0 & 0 & 1 \\
                       -1 & 0 & 0 \\
                       0 & -1 & 0 \\
%                       0 & -1 & 0 \\
%                       0 & 0 & -1 \\
%                       1 & 0 & 0 \\
                     \end{array}
                   \right).
$$
Thus the surface $\Xi_{(1,-1,-1)}^+$ is defined in $\P(1,1,1,1,2)$ by equations
$$
\left\{
  \begin{array}{ll}
    x_3=x_0-x_1-x_2, \\
    w=(\tau-2)(x_0x_1^2-x_1x_2^2-x_2x_0^2)+(\tau-3)x_0x_1x_2+(x_0^2x_1+x_1^2x_2-x_2^2x_0).
  \end{array}
\right.
$$
Therefore, the intersection $\Xi_{(1,1,1)}^+\cap \Xi_{(1,-1,-1)}^+$ is defined by equations
\begin{multline*}
\left\{
                                        \begin{array}{ll}
                                          x_1+x_2=0, \\
                                          (\tau-2)(x_1x_2^2+x_2x_0^2)=x_0^2x_1+x_1^2x_2.
                                        \end{array}
                                      \right.
\end{multline*}
This system of equations defines a set consisting of three points,
and~\mbox{$\Xi_{(1,1,1)}^+\bullet\Xi_{(1,-1,-1)}^+=0$}.
\end{example}

In the similar way one gets $\Xi_{v}^{+}=T_v(\Xi_{(1,1,1)}^{+})$,
where $T_v$ is an element in $\A_5$ given by Table~\ref{table:planes}.

\begin{table}[h]
\caption{The surfaces $\Xi_{v}^{+}$}\label{table:planes}
\begin{tabular}{|c|c|c|c|c|c|}
%\arraycolsep=1.4pt\def\arraystretch{1.5}
%\begin{array}{|c|c|c|c|c|c|}
  \hline
\vstrut  $v$ & $(1,1,1)$ & $(1,1,-1)$ & $(1,-1,1)$ & $(-1,1,1)$ & $(1,-1,-1)$ \\
\hline
\vstrut  $T_v$ & $\mathrm{Id}$  & $M^3$ & $R^2M^3$ & $RM^3$ & $RN$ \\
\hline\hline
\vstrut  $v$ & $(-1,1,-1)$ & $(-1,-1,1)$ & $(-1,-1,-1)$ & $(\tau-1, \tau,0)$ & $(1-\tau,\tau,0)$ \\
\hline
\vstrut  $T_v$ & $R^2N$ & $N$ & $M^2N$ & $M^4$ & $RM^2$ \\
\hline\hline
\vstrut  $v$ & $(1-\tau, -\tau, 0)$ & $(\tau-1,-\tau,0)$ & $(\tau, 0, 1-\tau)$ & $(\tau,0,\tau-1)$ & $(-\tau,0,\tau-1)$ \\
\hline
\vstrut  $T_v$ & $MN$ & $RM^4N$ & $M^2$ & $M$ & $M^4N$ \\
\hline\hline
\vstrut  $v$ & $(-\tau,0,1-\tau)$ & $(0,\tau-1,\tau)$ & $(0,\tau-1,-\tau)$ & $(0,1-\tau,-\tau)$ & $(0,1-\tau,\tau)$ \\
\hline
\vstrut  $T_v$ & $M^3N$ & $RM^4$ & $R^2M^4N$ & $RMN$ &$ R^2M^2$ \\
\hline
%\end{array}
\end{tabular}
\end{table}

This leads to the following result.

\begin{lemma}
\label{lemma:first-line}
One has
$$
\Xi_{(1,1,1)}^{+}\bullet\Xi_{v}^{+}=\left\{\aligned%
&\phantom{-2}0\ \text{if}\ v=(1,-1,-1), (-1,1,-1), (-1,-1,1), (-1,-1,-1),\\
&\phantom{-20\ \text{if}\ u=}\ (\tau-1,-\tau,0), (-\tau,0,\tau-1), (0,\tau-1,-\tau), \\
&\phantom{-2}1\ \text{if}\ v=(1,1,-1), (1,-1,1), (-1,1,1), (\tau-1, \tau,0), (1-\tau,\tau,0),\\
&\phantom{-21\ \text{if}\ v=}\ (1-\tau, -\tau, 0), (\tau, 0, 1-\tau), (\tau,0,\tau-1), (-\tau,0,1-\tau),\\
&\phantom{-21\ \text{if}\ v=}\ (0,\tau-1,\tau), (0,1-\tau,-\tau), (0,1-\tau,\tau),\\
&\,-2\ \text{if}\ v=(1,1,1).\\
\endaligned
\right.
$$
\end{lemma}

Similarly, we can compute all possible values of $\Xi_{v}^{+}\bullet\Xi_{v'}^{+}$.
They are given in Table~\ref{table:big} below.
The intersection matrix for the surfaces~$\Xi_{v}^{-}$ is the same as one given by Table~\ref{table:big}.

\begin{corollary}\label{corollary:14}
The $20\times 20$ matrix of intersection numbers of the surfaces $\Xi_{v}^{+}$,
where the index $v$ is taken from the list~\eqref{equation:v}, has rank~$14$.
Similarly, the $20\times 20$ matrix of intersection numbers of the surfaces $\Xi_{v}^{-}$ has rank~$14$.
\end{corollary}

\begin{proof}
Straightforward computation.
\end{proof}

Lemma~\ref{lemma:Stefan} and Corollary~\ref{corollary:14} imply the following result.

\begin{corollary}
\label{corollary:Xi-generate}
The classes of the $20$ surfaces $\Xi_{v}^{+}$,
where the index $v$ is taken from the list~\eqref{equation:v}, generate the $\Q$-vector space
$
\Cl(X)_{\Q}=\Cl(X)\otimes\Q.
$
Similarly, the classes of the $20$ surfaces $\Xi_{v}^{-}$
also generate the $\Q$-vector space $\Cl(X)_{\Q}$.
\end{corollary}

Corollaries~\ref{corollary:Cl-Q} and~\ref{corollary:Xi-generate} imply

\begin{proposition}
\label{proposition:Barth-class-group}
The group $\Cl(X)^G$ is generated by $-K_{X}$.
\end{proposition}

Now we derive Theorem~\ref{theorem:Barth} from Proposition~\ref{proposition:Barth-class-group}.

\begin{proof}[Proof of Theorem~\ref{theorem:Barth}]
Since $\Cl(X)^G$ is generated by $-K_{X}$ by Proposition~\ref{proposition:Barth-class-group},
the required assertion immediately follows from the proof of \cite[Theorem~A]{CheltsovPark}.
The only difference is that one should use a $G$-equivariant version of the standard Noether--Fano inequality,
which is~\mbox{\cite[Corollary~3.3.3]{CheltsovShramov}}.
\end{proof}

\begin{landscape}
\thispagestyle{empty}

{\fontsize{2.5}{4}\selectfont

\begin{table}
\caption{Intersection matrix for $\Xi_{v}^{+}$}\label{table:big}
%\begin{tabular}
\hspace*{-2.2cm}
\arraycolsep=1.4pt\def\arraystretch{4.5}
$\left.
  \begin{array}{|c|c|c|c|c|c|c|c|c|c|c|c|c|c|c|c|c|c|c|c|c|}
       \hline
      &\tabfont{(1,1,1)} &\tabfont{(1,1,-1)} &\tabfont{(1,-1,1)} &\tabfont{(-1,1,1)} &\tabfont{(1,-1,-1)} &\tabfont{(-1,1,-1)} &\tabfont{(-1,-1,1)} &\tabfont{(-1,-1,-1)} &\tabfont{(\tau-1, \tau,0)} &\tabfont{(1-\tau,\tau,0)} &\tabfont{(1-\tau, -\tau, 0)} &\tabfont{(\tau-1,-\tau,0)} &\tabfont{(\tau, 0, 1-\tau)} &\tabfont{(\tau,0,\tau-1)} &\tabfont{(-\tau,0,\tau-1)} &\tabfont{(-\tau,0,1-\tau)} &\tabfont{(0,\tau-1,\tau)} &\tabfont{(0,\tau-1,-\tau)} &\tabfont{(0,1-\tau,-\tau)} &\tabfont{(0,1-\tau,\tau)} \\
     \hline
    _{(1,1,1)}   & \mbox{\normalsize{$-2$}} & \mbox{\normalsize 1} & \mbox{\normalsize 1} & \mbox{\normalsize 1} & \mbox{\normalsize 0} & \mbox{\normalsize 0} & \mbox{\normalsize 0} & \mbox{\normalsize 0} & \mbox{\normalsize 1} & \mbox{\normalsize 1} & \mbox{\normalsize 1} & \mbox{\normalsize 0} & \mbox{\normalsize 1} & \mbox{\normalsize 1} & \mbox{\normalsize 0} & \mbox{\normalsize 1} & \mbox{\normalsize 1} & \mbox{\normalsize 0} & \mbox{\normalsize 1} & \mbox{\normalsize 1}   \\
     \hline
   _{(1,1,-1)}   & \mbox{\normalsize 1} & \mbox{\normalsize $-2$} & \mbox{\normalsize 0} & \mbox{\normalsize 0} & \mbox{\normalsize 1} & \mbox{\normalsize 1} & \mbox{\normalsize 0} & \mbox{\normalsize 0} & \mbox{\normalsize 1} & \mbox{\normalsize 1} & \mbox{\normalsize 1} & \mbox{\normalsize 0} & \mbox{\normalsize 1} & \mbox{\normalsize 1} & \mbox{\normalsize 1} & \mbox{\normalsize 0} & \mbox{\normalsize 0} & \mbox{\normalsize 1} & \mbox{\normalsize 1} & \mbox{\normalsize 1} \\
     \hline
   _{(1,-1,1)}   & \mbox{\normalsize 1} & \mbox{\normalsize 0} & \mbox{\normalsize $-2$} & \mbox{\normalsize 0} & \mbox{\normalsize 1} & \mbox{\normalsize 0} & \mbox{\normalsize 1} & \mbox{\normalsize 0} & \mbox{\normalsize 0} & \mbox{\normalsize 1} & \mbox{\normalsize 1} & \mbox{\normalsize 1} & \mbox{\normalsize 1} & \mbox{\normalsize 1} & \mbox{\normalsize 0} & \mbox{\normalsize 1} & \mbox{\normalsize 1} & \mbox{\normalsize 1} & \mbox{\normalsize 0} & \mbox{\normalsize 1}  \\
     \hline
   _{(-1, 1,1)}   & \mbox{\normalsize 1} & \mbox{\normalsize 0} & \mbox{\normalsize 0} & \mbox{\normalsize $-2$} & \mbox{\normalsize 0} & \mbox{\normalsize 1} & \mbox{\normalsize 1} & \mbox{\normalsize 0} & \mbox{\normalsize 1} & \mbox{\normalsize 1} & \mbox{\normalsize 0} & \mbox{\normalsize 1} & \mbox{\normalsize 1} & \mbox{\normalsize 0} & \mbox{\normalsize 1} & \mbox{\normalsize 1} & \mbox{\normalsize 1} & \mbox{\normalsize 0} & \mbox{\normalsize 1} & \mbox{\normalsize 1}  \\
     \hline
   _{(1,-1,-1)}   & \mbox{\normalsize 0} & \mbox{\normalsize 1} & \mbox{\normalsize 1} & \mbox{\normalsize 0} & \mbox{\normalsize $-2$} & \mbox{\normalsize 0} & \mbox{\normalsize 0} & \mbox{\normalsize 1} & \mbox{\normalsize 0} & \mbox{\normalsize 1} & \mbox{\normalsize 1} & \mbox{\normalsize 1} & \mbox{\normalsize 1} & \mbox{\normalsize 1} & \mbox{\normalsize 1} & \mbox{\normalsize 0} & \mbox{\normalsize 1} & \mbox{\normalsize 1} & \mbox{\normalsize 1} & \mbox{\normalsize 0} \\
     \hline
   _{(-1,1,-1)}   & \mbox{\normalsize 0} & \mbox{\normalsize 1} & \mbox{\normalsize 0} & \mbox{\normalsize 1} & \mbox{\normalsize 0} & \mbox{\normalsize $-2$} & \mbox{\normalsize 0} & \mbox{\normalsize 1} & \mbox{\normalsize 1} & \mbox{\normalsize 1} & \mbox{\normalsize 0} & \mbox{\normalsize 1} & \mbox{\normalsize 0} & \mbox{\normalsize 1} & \mbox{\normalsize 1} & \mbox{\normalsize 1} & \mbox{\normalsize 0} & \mbox{\normalsize 1} & \mbox{\normalsize 1} & \mbox{\normalsize 1}  \\
     \hline
   _{(-1,-1,1)}   & \mbox{\normalsize 0} & \mbox{\normalsize 0} & \mbox{\normalsize 1} & \mbox{\normalsize 1} & \mbox{\normalsize 0} & \mbox{\normalsize 0} & \mbox{\normalsize $-2$} & \mbox{\normalsize 1} & \mbox{\normalsize 1} & \mbox{\normalsize 0} & \mbox{\normalsize 1} & \mbox{\normalsize 1} & \mbox{\normalsize 1} & \mbox{\normalsize 0} & \mbox{\normalsize 1} & \mbox{\normalsize 1} & \mbox{\normalsize 1} & \mbox{\normalsize 1} & \mbox{\normalsize 0} & \mbox{\normalsize 1}  \\
     \hline
   _{(-1,-1,-1)}   & \mbox{\normalsize 0} & \mbox{\normalsize 0} & \mbox{\normalsize 0} & \mbox{\normalsize 0} & \mbox{\normalsize 1} & \mbox{\normalsize 1} & \mbox{\normalsize 1} & \mbox{\normalsize $-2$} & \mbox{\normalsize 1} & \mbox{\normalsize 0} & \mbox{\normalsize 1} & \mbox{\normalsize 1} & \mbox{\normalsize 0} & \mbox{\normalsize 1} & \mbox{\normalsize 1} & \mbox{\normalsize 1} & \mbox{\normalsize 1} & \mbox{\normalsize 1} & \mbox{\normalsize 1} & \mbox{\normalsize 0}  \\
     \hline
   _{(\tau-1, \tau,0)}   & \mbox{\normalsize  1} & \mbox{\normalsize 1} & \mbox{\normalsize 0} & \mbox{\normalsize 1} & \mbox{\normalsize 0} & \mbox{\normalsize 1} & \mbox{\normalsize 1} & \mbox{\normalsize 1} & \mbox{\normalsize $-2$} & \mbox{\normalsize 1} & \mbox{\normalsize 0} & \mbox{\normalsize 1} & \mbox{\normalsize 1} & \mbox{\normalsize 1} & \mbox{\normalsize 0} & \mbox{\normalsize 0} & \mbox{\normalsize 1} & \mbox{\normalsize 1} & \mbox{\normalsize 0} & \mbox{\normalsize 0}  \\
     \hline
   _{(1-\tau,\tau,0)}   & \mbox{\normalsize 1} & \mbox{\normalsize 1} & \mbox{\normalsize 1} & \mbox{\normalsize 1} & \mbox{\normalsize 1} & \mbox{\normalsize 1} & \mbox{\normalsize 0} & \mbox{\normalsize 0} & \mbox{\normalsize 1} & \mbox{\normalsize $-2$} & \mbox{\normalsize 1} & \mbox{\normalsize 0} & \mbox{\normalsize 0} & \mbox{\normalsize 0} & \mbox{\normalsize 1} & \mbox{\normalsize 1} & \mbox{\normalsize 1} & \mbox{\normalsize 1} & \mbox{\normalsize 0} & \mbox{\normalsize 0}  \\
     \hline
   _{(1-\tau, -\tau, 0)}  & \mbox{\normalsize 1} & \mbox{\normalsize 1} & \mbox{\normalsize 1} & \mbox{\normalsize 0} & \mbox{\normalsize 1} & \mbox{\normalsize 0} & \mbox{\normalsize 1} & \mbox{\normalsize 1} & \mbox{\normalsize 0} & \mbox{\normalsize 1} & \mbox{\normalsize $-2$} & \mbox{\normalsize 1} & \mbox{\normalsize 0} & \mbox{\normalsize 0} & \mbox{\normalsize 1} & \mbox{\normalsize 1} & \mbox{\normalsize 0} & \mbox{\normalsize 0} & \mbox{\normalsize 1} & \mbox{\normalsize 1} \\
     \hline
   _{(\tau-1,-\tau,0)}   & \mbox{\normalsize 0} & \mbox{\normalsize 0} & \mbox{\normalsize 1} & \mbox{\normalsize 1} & \mbox{\normalsize 1} & \mbox{\normalsize 1} & \mbox{\normalsize 1} & \mbox{\normalsize 1} & \mbox{\normalsize 1} & \mbox{\normalsize 0} & \mbox{\normalsize 1} & \mbox{\normalsize $-2$} & \mbox{\normalsize 1} & \mbox{\normalsize 1} & \mbox{\normalsize 0} & \mbox{\normalsize 0} & \mbox{\normalsize 0} & \mbox{\normalsize 0} & \mbox{\normalsize 1} & \mbox{\normalsize 1}  \\
     \hline
    _{(\tau, 0, 1-\tau)}  & \mbox{\normalsize 1} & \mbox{\normalsize 1} & \mbox{\normalsize 1} & \mbox{\normalsize 1} & \mbox{\normalsize 1} & \mbox{\normalsize 0} & \mbox{\normalsize 1} & \mbox{\normalsize 0} & \mbox{\normalsize 1} & \mbox{\normalsize 0} & \mbox{\normalsize 0} & \mbox{\normalsize 1} & \mbox{\normalsize $-2$} & \mbox{\normalsize 1} & \mbox{\normalsize 0} & \mbox{\normalsize 1} & \mbox{\normalsize 0} & \mbox{\normalsize 1} & \mbox{\normalsize 1} & \mbox{\normalsize 0}  \\
     \hline
    _{(\tau,0,\tau-1)}  & \mbox{\normalsize 1} & \mbox{\normalsize 1} & \mbox{\normalsize 1} & \mbox{\normalsize 0} & \mbox{\normalsize 1} & \mbox{\normalsize 1} & \mbox{\normalsize 0} & \mbox{\normalsize 1} & \mbox{\normalsize 1} & \mbox{\normalsize 0} & \mbox{\normalsize 0} & \mbox{\normalsize 1} & \mbox{\normalsize 1} & \mbox{\normalsize $-2$} & \mbox{\normalsize 1} & \mbox{\normalsize 0} & \mbox{\normalsize 1} & \mbox{\normalsize 0} & \mbox{\normalsize 0} & \mbox{\normalsize 1}  \\
     \hline
   _{(-\tau,0,\tau-1)}   & \mbox{\normalsize 0} & \mbox{\normalsize 1} & \mbox{\normalsize 0} & \mbox{\normalsize 1} & \mbox{\normalsize 1} & \mbox{\normalsize 1} & \mbox{\normalsize 1} & \mbox{\normalsize 1} & \mbox{\normalsize 0} & \mbox{\normalsize 1} & \mbox{\normalsize 1} & \mbox{\normalsize 0} & \mbox{\normalsize 0} & \mbox{\normalsize 1} & \mbox{\normalsize $-2$} & \mbox{\normalsize 1} & \mbox{\normalsize 1} & \mbox{\normalsize 0} & \mbox{\normalsize 0} & \mbox{\normalsize 1}  \\
     \hline
    _{(-\tau,0,1-\tau)}  & \mbox{\normalsize 1} & \mbox{\normalsize 0} & \mbox{\normalsize 1} & \mbox{\normalsize 1} & \mbox{\normalsize 0} & \mbox{\normalsize 1} & \mbox{\normalsize 1} & \mbox{\normalsize 1} & \mbox{\normalsize 0} & \mbox{\normalsize 1} & \mbox{\normalsize 1} & \mbox{\normalsize 0} & \mbox{\normalsize 1} & \mbox{\normalsize 0} & \mbox{\normalsize 1} & \mbox{\normalsize $-2$} & \mbox{\normalsize 0} & \mbox{\normalsize 1} & \mbox{\normalsize 1} & \mbox{\normalsize 0}  \\
     \hline
    _{(0,\tau-1,\tau)}  & \mbox{\normalsize 1} & \mbox{\normalsize 0} & \mbox{\normalsize 1} & \mbox{\normalsize 1} & \mbox{\normalsize 1} & \mbox{\normalsize 0} & \mbox{\normalsize 1} & \mbox{\normalsize 1} & \mbox{\normalsize 1} & \mbox{\normalsize 1} & \mbox{\normalsize 0} & \mbox{\normalsize 0} & \mbox{\normalsize 0} & \mbox{\normalsize 1} & \mbox{\normalsize 1} & \mbox{\normalsize 0} & \mbox{\normalsize $-2$} & \mbox{\normalsize 1} & \mbox{\normalsize 0} & \mbox{\normalsize 1}  \\
     \hline
    _{(0,\tau-1,-\tau)}  & \mbox{\normalsize 0} & \mbox{\normalsize 1} & \mbox{\normalsize 1} & \mbox{\normalsize 0} & \mbox{\normalsize 1} & \mbox{\normalsize 1} & \mbox{\normalsize 1} & \mbox{\normalsize 1} & \mbox{\normalsize 1} & \mbox{\normalsize 1} & \mbox{\normalsize 0} & \mbox{\normalsize 0} & \mbox{\normalsize 1} & \mbox{\normalsize 0} & \mbox{\normalsize 0} & \mbox{\normalsize 1} & \mbox{\normalsize 1} & \mbox{\normalsize $-2$} & \mbox{\normalsize 1} & \mbox{\normalsize 0} \\
     \hline
    _{(0,1-\tau,-\tau)}  & \mbox{\normalsize 1} & \mbox{\normalsize 1} & \mbox{\normalsize 0} & \mbox{\normalsize 1} & \mbox{\normalsize 1} & \mbox{\normalsize 1} & \mbox{\normalsize 0} & \mbox{\normalsize 1} & \mbox{\normalsize 0} & \mbox{\normalsize 0} & \mbox{\normalsize 1} & \mbox{\normalsize 1} & \mbox{\normalsize 1} & \mbox{\normalsize 0} & \mbox{\normalsize 0} & \mbox{\normalsize 1} & \mbox{\normalsize 0} & \mbox{\normalsize 1} & \mbox{\normalsize $-2$} & \mbox{\normalsize 1} \\
     \hline
    _{(0,1-\tau,\tau)}  & \mbox{\normalsize 1} & \mbox{\normalsize 1} & \mbox{\normalsize 1} & \mbox{\normalsize 1} & \mbox{\normalsize 0} & \mbox{\normalsize 1} & \mbox{\normalsize 1} & \mbox{\normalsize 0} & \mbox{\normalsize 0} & \mbox{\normalsize 0} & \mbox{\normalsize 1} & \mbox{\normalsize 1} & \mbox{\normalsize 0} & \mbox{\normalsize 1} & \mbox{\normalsize 1} & \mbox{\normalsize 0} & \mbox{\normalsize 1} & \mbox{\normalsize 0} & \mbox{\normalsize 1} & \mbox{\normalsize $-2$}  \\
           \hline
  \end{array}
\right.$
%\end{tabular}
\end{table}
}
\end{landscape}


\newpage

\begin{thebibliography}{19}

\bibitem[Ba96]{Barth}
W.~Barth, \emph{Two projective surfaces with many nodes, admitting
the symmetries of the icosahedron}, J. Alg. Geom.
\textbf{5} (1996), 173--186.

\bibitem[Bu91]{Burkhardt}
H.~Burkhardt, \emph{Untersuchungen aus dem Gebiete der hyperelliptischen Modulfunctionen}, Math. Ann. \textbf{38} (1891), 161--224.

\bibitem[CP10]{CheltsovPark}
I.~Cheltsov, J.~Park, \emph{Sextic double solids},
Cohomological and Geometric Approaches to Rationality Problems,
Progress in Mathematics \textbf{282} (2010), 75--132.

\bibitem[CPS16]{CheltsovPrzyjalkowskiShramov}
I.~Cheltsov, V.~Przyjalkowski, C.~Shramov, \emph{Quartic double solids with icosahedral symmetry}, European J. of Math. \textbf{2} (2016), no.~1, 96--119.

\bibitem[CS14]{ChSh09b}
I.~Cheltsov, C.~Shramov, \emph{Five embeddings of one simple group}, Trans. of the AMS \textbf{366} (2014), 1289--1331.

\bibitem[CS15]{CheltsovShramov}
I.~Cheltsov, C.~Shramov, \emph{Cremona groups and the icosahedron}, CRC Press, 2015.

\bibitem[Co06]{Coble}
A.~Coble, \emph{An invariant condition for certain automorphic algebraic forms}, Amer. J. Math. \textbf{28} (1906), 333--366.

\bibitem[Co00]{Co00}
A.~Corti, \emph{Singularities of linear systems and 3-fold birational geometry}, L.M.S. Lecture Note Series \textbf{281} (2000), 259--312.

\bibitem[$\mathrm C^+$]{Atlas}
J.~Conway, R.~Curtis, S.~Norton, R.~Parker, R.~Wilson, \emph{Atlas of finite groups},  Clarendon Press, Oxford, 1985.

\bibitem[En99]{En99}
S.~Endra\ss, \emph{On the divisor class group of double solids}, Manuscripta Mathematica \textbf{99} (1999), 341--358.

\bibitem[GAP]{GAP}
The GAP Group, GAP --- Groups, Algorithms, and Programming, Version 4.7.4; 2014.

\bibitem[Hu96]{Hunt}
B.~Hunt, \emph{The geometry of some special arithmetic quotients}, Lecture Notes in Mathematics, \textbf{1637}. Berlin: Springer, 1996.

\bibitem[IP99]{IsPr99}
V.~Iskovskikh, Yu. Prokhorov, \emph{Fano Varieties}, Encyclopaedia of Mathematical Sciences \textbf{47} (1999), Springer, Berlin.

\bibitem[JR97]{JR}
D.~Jaffe, D.~Ruberman, \emph{A sextic surface cannot have 66 nodes}, J. Alg. Geom. \textbf{6:1} (1997), 151--168.

\bibitem[JS-BV90]{deJong}
A.~de Jong, N.~Shepherd-Barron, A.~Van de Ven, \emph{On the Burkhardt quartic}, Math. Ann. \textbf{286} (1990), 309--328.

\bibitem[Ka11]{Kaloghiros}
A.-S.~Kaloghiros, \emph{The defect of Fano $3$-folds}, J. Alg. Geom. \textbf{20:1} (2011), 127--149.

\bibitem[Me04]{Mella}
M.~Mella, \emph{Birational geometry of quartic 3-folds. II. The importance of being $\mathbb{Q}$-factorial}, Math. Ann. \textbf{330} (2004), 107--126.

\bibitem[Pr13]{Prokhorov1}
Yu.~Prokhorov, \emph{$G$-Fano threefolds, I, II}, Adv. Geom. 13 (2013), no. 3, 389--418, 419--434.

\bibitem[Pr17]{Prokhorov-factorial}
Yu.~Prokhorov, \emph{On the number of singular points of factorial terminal Fano
threefolds}, Mathematical Notes, 2017, 101:6, 1068--1073.

\bibitem[PS16]{PrzyalkowskiShramov}
V.~Przyjalkowski, C.~Shramov, \emph{Double quadrics with large automorphism groups},
Proc. Steklov Inst. Math., 294 (2016), 154--175.
%to appear in  Proc. Steklov Inst. Math, \textbf{294} (2016), arXiv:1604.00307.

\bibitem[To35]{Todd}
J.~Todd, \emph{On a quartic primal with forty-five nodes, in space of four dimensions}, Q. J. Math. \textbf{7} (1935), 168--174.

\end{thebibliography}
\end{document}